\documentclass[12pt]{article}
\usepackage{amssymb,amsmath,bm}

\usepackage[colorlinks=true, urlcolor=blue,linkcolor=blue, citecolor=blue]{hyperref}
\usepackage{mathrsfs,verbatim}
\usepackage{caption,cleveref }
\usepackage{graphicx, float}
\usepackage{color, mathtools, tikz, xcolor}
\usepackage{enumerate}
\usepackage{tikz}
\usetikzlibrary{shapes.geometric}
\usepackage{pdfpages}

\usepackage{caption}

\usetikzlibrary{positioning}
\usetikzlibrary{arrows}

\usepackage{subcaption}

\usetikzlibrary{calc}

\DeclareMathOperator{\gpf}{gpf}

\newcommand{\bea}{\begin{eqnarray}}
\newcommand{\ena}{\end{eqnarray}}
\newcommand{\beas}{\begin{eqnarray*}}
\newcommand{\enas}{\end{eqnarray*}}
\newcommand{\beq}{\begin{equation}}
\newcommand{\enq}{\end{equation}}

\newcommand{\ignore}[1]{}

\newtheorem{theorem}{Theorem}[section]
\newtheorem{corollary}[theorem]{Corollary}

\newtheorem{lemma}[theorem]{Lemma}
\newtheorem{definition}[theorem]{Definition}

\usepackage{lipsum}
\usepackage{enumerate}
\usepackage{amssymb}
\usepackage{amsmath}

\numberwithin{equation}{section}

\begin{document}

\title{ Large Zsigmondy Primes}


\author{Ömer Avcı\footnote{Department of Mathematics, Bo\u{g}azi\c{c}i University, 34342, Bebek, Istanbul, Turkey}}

\maketitle
\begin{abstract}
If $a>b$ and $n>1$ are positive integers and $a$ and $b$ are relatively prime integers, then a \emph{large Zsigmondy prime} for $(a,b,n)$ 
is a
prime $p$ such that $p \,|\, a^n-b^n$ but $p \,\nmid \, a^m-b^m$
  for $1 \leq m < n$ and either $p^2 \, | \, a^n - b^n$ or
$ p > n + 1$. We classify all the triples of integers $(a, b, n)$ for which no large Zsigmondy prime exists.

\end{abstract}

\section{Introduction}

Let $a>b$ be relatively prime positive integers and $n$ be a positive integer. A Zsigmondy prime for
$(a, b, n)$ is defined as a prime $p$ such that $p \mid a^n-b^n$ but $p \nmid a^m-b^m$ for $1 \leq m < n$. Zsigmondy's Theorem asserts that Zsigmondy primes exist for all triples $(a,b,n)$ except when $(a,b, n)$ = $(2,1,6)$ or $n = 2$ and $a +b = 2^k$ for some positive integer $k$ (see \cite{zsigmo}). Zsigmondy's Theorem was independently, but later, discovered by Birkhoff and Vandiver \cite{vandiver}. \\

In \cite{feit}, Feit deals with the special case of Zsigmondy's Theorem when $b=1$ and defines a large Zsigmondy prime for the pair $(a,n)$ as a 
prime $p$ such that $p \mid a^n-1$ but $p \nmid a^m-1$ for $1 \leq m < n$ and either $p^2 \mid a^n -1$ or
$ p > n + 1$. \\

In our paper, we present a generalised version of Feit's results. 

\begin{theorem} \label{Main}
If $a>b$ are relatively prime positive integers and $n$ is an integer greater than 1, then there exists a large Zsigmondy prime for $(a,b,n)$ except the following cases.
\begin{enumerate}[(i)]
\item  $n=2$ and $a+b= 2^s$ or $a+b=3\cdot2^s$ where $s$ is a non-negative integer. 
\item  $n=4$ and $(a,b)$ is $(2,1)$ or $(3,1)$. 
\item  $n=6$ and  $(a,b)$ is one of the following $\{(2,1),(3,1),(3,2),(5,4)\}$. 
\item  $n \in \{10,12,18\}$ and $(a,b)=(2,1)$. 

\end{enumerate}
\end{theorem} 
Artin's results about orders of linear groups (see 
\cite{artin}) also inspired Feit's work about the existence of 
large Zsigmondy primes. The motivation for Feit's work comes 
from the theory of finite groups \cite{feit3}. Feit proved the
existence of large Zsigmondy primes in all cases except for finitely many, as stated in \cite{feit2}, for the special case $a \geqslant 3$. Later on,
he came up with a simpler proof of his result, which also 
includes the case where $a=2$, as presented in \cite{feit}.
Roitman also provided a nice proof of Feit’s result in \cite{roitman}. 

For relatively prime positive integers $a>b$, we can 
generalize the definition of a large Zsigmondy prime as a 
prime $p$ such that $p \mid a^n-b^n$, but $p \nmid a^m-b^m$
  for $1 \leq m < n$ and either $p^2 \mid a^n -b^n$ or
$ p > n + 1$. 
We show that there exists a large Zsigmondy prime for $(a, b, n)$  except in the cases presented in Theorem \ref{Main}. Our proof is inspired by the elegant proof 
of Zsigmondy's Theorem given by Yan Sheng in \cite{yansheng}.

\section{Preliminaries}

\begin{lemma}\label{sqrtineq} \cite{feit}
For any positive integer 
$n$, where 
$\phi(n)$ denotes Euler's totient function, it holds that:
$$\phi(n) \geqslant \frac{1}{2} \sqrt{n}.$$
\end{lemma}

\begin{lemma}[Lifting the Exponent Lemma - LTE]
For a prime $p$ and a positive integer $n$, let $v_p(n)$ denote the exponent of $p$ in the prime factorisation of $n$. 
Let $x$ and $y$ be integers such that 
$x \equiv y\not\equiv0 \pmod p.$
\begin{enumerate}[(1)]
\item If $p\geqslant3$, then
\begin{align*}
v_p(x^n-y^n) = v_p(x-y) + v_p(n).
\end{align*}\\
\item If $p=2$, then \begin{align*}
    v_2(x^n-y^n)=
    \begin{cases}
       v_2(x^2-y^2)+v_2(n)-1, & \text{if n is even}\\\
      v_2(x-y), & \text{if n is odd.}
    \end{cases}
  \end{align*}

\end{enumerate}

\end{lemma}

\begin{definition} (Cyclotomic Polynomials)
For any positive integer 
$n$, the 
$n$-th cyclotomic polynomial 
$\Phi_n(x)$ is given by:
\begin{equation*}
    \Phi_n(x) = \prod_{\substack{\gcd(k,n)=1 \\ 1\leqslant k \leqslant n} } (x-e^{2 i \pi \frac{ k}{n}}).
\end{equation*}
It is known that $\Phi_n(x)$ is a monic polynomial with integer coefficients.

\end{definition}
\begin{definition}\label{tanim}
There is a generalization of cyclotomic polynomials into two variables: 
\begin{equation*}
\Phi_n(a,b) = b^{\phi(n)} \Phi_n \Big(\dfrac{a}{b} \Big).
\end {equation*}
We can also express $\Phi_n(a,b)$ as
\begin{equation*}
\Phi_n(a,b) = \prod_{\substack{\gcd(k,n)=1 \\ 1\leqslant k \leqslant n} } (\,a-b \, e^{2 i \pi \frac{ k}{n}} \,).
\end {equation*}
It is clear that $\Phi_n(x,y)$ is a two variable polynomial with integer coefficients.
\end{definition}

\begin{lemma}\cite{yansheng} \label{cyclofactorization}
 Let $n$ be a positive integer. Then,
\begin{equation*}
x^n-1 = \prod_{d \mid n} \Phi_d(x). 
\end {equation*}
\end{lemma}

\begin{corollary} \label{productequality}  \cite{yansheng}
Let $a,b,n$ are positive integers. Then,
\begin{equation*}
a^n-b^n = \prod_{d \mid n} \Phi_d(a,b). 
\end {equation*}

\end{corollary}

\begin{lemma}\label{inequality} \cite{yansheng}
Let $p$ be a prime, $n\geqslant3$ be an integer and $x>0$. Then,
\begin{equation*}
(x-1)^{\phi(n)}<\Phi_{n}(x)<(x+1)^{\phi(n)}.
\end {equation*}
\end{lemma}

\begin{corollary} \label{cor_ineq}
Let $a,b,n$ are positive integers and $n\geqslant3$. Then,
\begin{equation*}
(a-b)^{\phi(n)}<\Phi_{n}(a,b)<(a+b)^{\phi(n)}.
\end {equation*}
\end{corollary}

\begin{lemma} \cite{yansheng}
Let $p$ be a prime, $a$ and $b$ be distinct positive integers, and $n$ be a positive integer. Then,
\begin{equation*}
\Phi_{pn}(a,b) = \begin{cases}
      \Phi_{n}(a^p,b^p) & \text{ if $p \mid n$}\\\
      \frac{\Phi_{n}(a^p,b^p)}{\Phi_{n}(a,b)} & \text{ if $p \nmid n$}
    \end{cases}
\end {equation*}
\end{lemma}

\begin{corollary}\label{imp_cor}
\cite{yansheng}
Let $p$ be a prime, $a$ and $b$ be distinct positive integers, and $n=p^\beta k$ for some positive integers $\beta , k $ with $p \nmid k$. Then,
\begin{equation*}
\Phi_{n}(a,b) = 
      \Phi_{p k}(a^{p^{\beta -1}},b^{p^{\beta -1}}) =
      \frac{\Phi_{k}(a^{p^{\beta }},b^{p^{\beta }})}{\Phi_k(a^{p^{\beta-1 }},b^{p^{\beta-1}})}.
\end {equation*}
\end{corollary}

\begin{lemma}\label{lte_p} \cite{yansheng}
Let $p$ be a prime, $a$ and $b$ be distinct positive integers not divisible to $p$, and $n$ be a positive integer. Let $k$ be the smallest positive integer satisfying $p \mid a^k-b^k$. Then,
\begin{equation*}
v_p(\Phi_{n}(a,b)) = \begin{cases}
      v_p(a^k-b^k) & \text{$n=k$,}\\\
      1 & \text{$n=p^\beta k,\, \beta \geqslant1  $, } \\\
      0 & \text{otherwise.}
    \end{cases}
\end {equation*}
\end{lemma}

\begin{lemma}\label{lte_2} \cite{yansheng}
Let $a$ and $b$ be distinct odd positive integer and $n$ be a positive integer. Then,
\begin{equation*}
v_2\left(\Phi_{n}(a,b)\right) = 
    \begin{cases}
        v_2(a-b) & \text{$n=1$}\\\
        v_2(a+b) & \text{$n=2$}\\\
        1       & \text{$n=2^\beta,\, \beta \geqslant2  $ } \\\
        0       & \text{else.}
    \end{cases}
\end {equation*}
\end{lemma}

\section{Results on Zsigmondy Primes}

\begin{lemma}\label{cyclotomic_cases}
Let $a>b$ be two relatively prime positive integers and  $n$ be a positive integer, $p$ be a prime divisor of $\Phi_n(a,b)$ and $k$ be the smallest positive integer satisfying $p \mid a^k-b^k$. Let $\gpf(n)$ denote the largest prime divisor of $n$, then one of the following holds:
\begin{enumerate}[(i)]
\item $p=2$ and $n=2^\beta $ for some $\beta \geqslant 1.$
\item  $p \geq 3$ and $n = k$ thus $p$ is a Zsigmondy prime for $(a,b,n).$
\item $p=\gpf(n) >2$ and $n=p^\beta k$ for some $\beta \geqslant 1$ and $v_p(\Phi_n(a,b))=1.$

\end{enumerate}

\end{lemma}

\begin{proof}
If $p=2$, by Lemma ~\ref{lte_2}, it follows that $n=2^\beta $ for some $\beta \geqslant 1$ .
If $p>2$, according to  Lemma ~\ref{lte_p}, there are two possibilities. Either $n=k$ or $n=p^\beta k$ holds.
When $n=k$, it implies that $p \nmid a^m-b^m$ for all $1 \leq m < n$, which means that $p$ is a Zsigmondy prime for $(a,b,n).$ 
Since $k$ is defined as the smallest positive integer such that 
 $p|a^k-b^k$, it is evident that $k\mid p-1$ holds. Moreover, it is clear that any prime divisor of $k$  must be smaller than $p$. Consequently, when $n=p^\beta k$, we can conclude that $p=\gpf(n)$ since no prime divisor of $n$ can be greater than $p$. Furthermore, according to Lemma \ref{lte_p}, we have $v_p(\Phi_n(a,b))=1$ in the case where $n=p^\beta k.$

\end{proof}

\begin{lemma} \label{cyclo_importance}
Let $a$ and $b$ be distinct, relatively prime positive integers and  $n\geqslant2$ be an integer. If $p$ is a Zsigmondy prime for $(a, b, n)$ then $p \mid \Phi_n(a,b)$.
\end{lemma}

\begin{proof}
From Corollary ~\ref{productequality} we have 
\begin{equation*}
    a^n-b^n=
    \prod_{d\mid n}\Phi_d(a,b).
\end{equation*}
Therefore, such $p$ divides $\Phi_d(a,b)$ for some $d \mid n$. If $d < n$, then
    $p \mid \Phi_d(a,b)$, which implies  $p|a^d-b^d$.
This contradicts with $p$ being a Zsigmondy prime for $(a,b,n)$. We conclude that $d=n$, hence $ p \mid \Phi_n(a,b) $.

\end{proof}

\begin{lemma}
Let $a>b$ be relatively prime positive integers, and  $n\geqslant2$ be an integer. 
If $q$ is a Zsigmondy prime for $(a,b,n)$ but not a large Zsigmondy prime for $(a,b,n)$, then $n=q-1$.

\end{lemma}

\begin{proof}
Since $q$ is a Zsigmondy prime for $(a,b,n)$, $n \mid q-1;$ but since  $q$ is not large Zsigmondy prime for $(a,b,n)$, $q \leqslant n + 1$. Therefore $n=q-1$.

\end{proof}

\begin{lemma} \label{large_zsigmondy_existence}
Let $a$ and $b$ be distinct, relatively prime, positive integers, and  $n\geqslant3$ be an integer. Then there is a large Zsigmondy prime for $(a,b,n)$, if $(n+1) \gpf(n) <  \Phi_n(a,b).$

\end{lemma}

\begin{proof}
 Let us analyze the proof in two cases. 

 If $\Phi_n(a,b)$ is even, then $n=2^\beta$ for some $\beta\geq 2$ and $4\nmid \Phi_n(a,b)$ from \ref{lte_2}. Since $\Phi_n(a,b)>2(n+1)>2$, it has at least one odd prime divisor. Let $p$ be the greatest prime divisor of $\Phi_n(a,b)$. Since $p>2$ and $p\nmid n$, we obtain $n|p-1$ from \ref{lte_p}. If $p>n+1$, then $p$ is a large Zsigmondy prime for $(a,b,n).$ If $p=n+1$, then only odd prime divisor of $\Phi_n(a,b)$ is $p$. Since $4\nmid \Phi_n(a,b)$ and $\Phi_n(a,b)>2(n+1)$ we conclude that $p^2|\Phi_n(a,b),$ and therefore $p$ is a large Zsigmondy prime for $(a,b,n).$

 If $\Phi_n(a,b)$ is odd, it must have an odd prime divisor. Let $p$ be the greatest prime 
 divisor of $\Phi_n(a,b).$ From \ref{lte_p}, we obtain $n|p-1$ or $p|n$. If $n|p-1$ and 
 $p>n+1$ then $p$ is a large Zsigmondy prime for $(a,b,n).$ If $p=n+1$ and $p^2|\Phi_n(a,b)$ 
 then it is a large Zsigmondy prime for $(a,b,n).$ If $p=n+1$ and $p^2\nmid \Phi_n(a,b)$, 
 then $\Phi_n(a,b)$ must have another prime divisor $q$ because of the assumption
 $\Phi_n(a,b)>(n+1)\gpf(n).$ 
 Since $q<p$, this implies that $n\nmid q-1$, therefore $q$ is the greatest prime divisor of $n$ 
 because of \ref{lte_p}, and furthermore $q^2\nmid \Phi_n(a,b).$ Thus only prime divisors of 
 $\Phi_n(a,b)$ are $p$ and $q$. Also, their squares does not divide $\Phi_n(a,b)$. Ultimately $\Phi_n(a,b)=pq$ must hold but it contradicts with $\Phi_n(a,b)>(n+1)\gpf(n)$ since $p=n+1$ and $q=\gpf(n).$

\end{proof}

\vspace{5mm}

\begin{mainproof}

We begin by proving the existence of a large Zsigmondy prime for $(a,b,n)$, when $n$ is not equal to any of the numbers $\{2,4,6,10,12,18\}$.
Consider positive integers $a>b$ and $n>1$ with $\gcd(a,b)=1$. Let's assume that there is no large Zsigmondy prime for $(a,b,n)$. If there is no Zsigmondy prime for $(a,b,n)$, we can determine the possible values of $(a,b,n)$ based on Zsigmondy's theorem. We will specifically investigate the case where there is a Zsigmondy prime for $(a,b,n)$ but no large Zsigmondy prime for $(a,b,n).$

Let $n \geqslant 3$ and $q$ be a Zsigmondy prime for $(a,b,n)$ but $q$ is not a large Zsigmondy prime for $(a,b,n)$; therefore, 
$n=q-1$ and $q^2 \nmid a^n-b^n$. From Lemma \ref{cyclo_importance} we know that it is necessary for $q \mid \Phi_n(a,b)$ to hold. From Lemma \ref{cyclotomic_cases}, $\Phi_n(a,b)$ can have at most one non-Zsigmondy prime divisor $p$ with the possibilities $p=2$ or $p=\gpf(n)$. Now, we have three cases to consider:
\begin{enumerate}[(i)]
    \item $\Phi_n(a,b)=2 q$ and $n = 2^\beta$ where $\beta \geqslant 2$. In this case we have $q=2^\beta+1$ therefore it must be a Fermat prime so $\beta=2^s$ for some $s \geqslant1$. From Corollary ~\ref{imp_cor} we have
    \begin{equation*}
        \Phi_n(a,b)= 
        \Phi_2(a^{2^{\beta -1}},b^{2^{\beta -1}})=
        a^{2^{\beta -1}}+b^{2^{\beta -1}} \geqslant
        2^{2^{\beta -1}}+1. 
    \end{equation*}
For $\beta \geqslant 4$ we have $ 2^{2^{\beta -1}}+1 > 2 (2^\beta +1) $ therefore $\Phi_n(a,b) > 2(n+1) = 2 q$, leading to a contradiction with our assumption. We are left with two possibilities: $n=4$ or $n=8$. However, if $n=8$ then $q=n+1$ cannot be a prime. Therefore, the only possibility in this case is $n=4$.

\item $\Phi_n(a,b)=pq$, where $p=\gpf(n)>2$, is the greatest prime divisor of $n$. Then $n = p^\beta k$ where $\beta$ is a positive integer and $k$ is the smallest positive integer satisfying $p \mid a^k-b^k$. Clearly, $k \mid p-1$. \\ We can divide this case into two subcases.
\begin{enumerate}
    \item 

If $\beta \geqslant  2$, then by combining Corollary ~\ref{imp_cor} and Corollary ~\ref{cor_ineq}, we can get:
\begin{equation*}
    \Phi_n(a,b)= 
        \Phi_{pk}(a^{p^{\beta -1}},b^{p^{\beta -1}}) \geqslant
        (a^{p^{\beta -1}}-b^{p^{\beta -1}})^{\Phi(pk)}.
\end{equation*}
 Since $a>b$, we can derive the inequality,
\begin{equation*}
        (a^{p^{\beta -1}}-b^{p^{\beta -1}})^{\Phi(pk)} \geqslant
        (2^{p^{\beta -1}}-1)^{\Phi(pk)} \geqslant
        (2^{p^{\beta -1}}-1)^{p-1} \geqslant
         (2^{p-1}-1)^{p^{\beta -1}}.
\end{equation*} 
Since $k<p$  we have,
 \begin{equation*}
     \Phi_n(a,b)=
     p q =
     p (p^\beta k +1)  <
     p^{\beta +2 }. 
 \end{equation*}
Since $p \geqslant 3$, we have $2^{p -1} - 1 \geqslant p$, thus,
$$\Phi_n(a,b)\geq (2^{p-1}-1)^{p^{\beta-1}}\geq p^{p^{\beta-1}}.$$
Therefore $ \beta +2 > p^{\beta-1} $ must hold, which is not possible when 
$\beta \geqslant 3$. Therefore, $\beta \neq 2$, then a large Zsigmondy prime exists for $(a,b,n)$ in this case. Let's investigate the case $\beta = 2$. By substituting $\beta=2$ into our previous inequalities, we obtain,
\begin{equation*}
   p^4=p^{\beta +2} > \Phi_n(a,b) \geqslant
     (2^{p^{\beta -1}}-1)^{p-1}=
      (2^p-1)^{p-1} \geqslant
        (2^p-1)^2.
\end{equation*}
It is not possible when $p\geq 5$. Therefore, there exists a large Zsigmondy prime for $(a,b,n)$ when $p\geq 5$ in this case. So, in the second case, if there is no large Zsigmondy prime for $(a,b,n)$, then $p=3$, $\beta =2$, and $k=1$ or $k=2$.
Thus, the only exceptional values are $n=18$ and $n=9$. If $n=9$ then $n+1$ is not a prime, and $q=n+1$ is not a Zsigmondy prime for $(a,b,n)$. Therefore, the only possibility in this case is $n=18$. We will find the pairs $(a,b)$ at the end of the proof. 
\item If $\beta=1$, then by combining Corollary ~\ref{imp_cor} and Corollary ~\ref{cor_ineq}, we can obtain,
\begin{equation*}
    \Phi_n(a,b)=
    \Phi_{pk}(a,b)=
    \frac{\Phi_k(a^p,b^p)}{\Phi_k(a,b)} \geqslant
    \Big(\frac{a^p-b^p}{a+b}\Big)^{\phi(k)} \geqslant
     \Big(\frac{2^p-1}{3}\Big)^{\phi(k)}.
\end{equation*}
 In this case, $\Phi_n(a,b)=(p k + 1) p < p^3$ holds. Then either 
 $\frac{2^p-1}{3} < p$ or $\phi(k) < 3$. Which means either
  $p \leqslant 3$ or $k \leqslant6$. If $p=3$ then $n=6$. If $p>3$ then $\phi(k)\leq 2$ thus $k\in\{1,2,3,4,6\}$. 
  
  If $\phi(k)=2$ then $k\in\{3,4,6\}$ and
  $$p^3>\Phi_n(a,b)\geq \Big(\frac{2^p-1}{3}\Big)^2$$
holds which is not possible when $p\geq 7$. If $p=5$ then $k=4$ must hold since $k|p-1.$ But then $n=20$ so $q=n+1$ is not a Zsigmondy prime for $(a,b,n).$

If $\phi(k)=1$ then $k\in\{1,2\}$ and
$$p^3>\Phi_n(a,b)\geq \frac{2^p-1}{3}$$
holds which is not possible when $p\geq 13.$ If $k=1$ then $n=p$ holds. But then $q=n+1$ can not be a prime number. If $k=2$ then $n=2p$ holds. If $p=11$ then $q=23$ and $\Phi_n(a,b)=253.$ But this contradicts with the fact $\Phi_{10}(a,b)\geq \frac{2^{10}-1}{3} = 341.$ If $p=7$ then $n=14$ but then $q=n+1$ is not a prime number. Then $p=5$ and $n=10$ must hold. 

Ultimately only possible values are $n=6$ and $n=10$ in this case.
 Again we will handle the determination of pairs $(a,b)$ at the end of the proof.

\end{enumerate}

\item
$\Phi_n(a,b)=q$, where $q=n+1$, is an odd prime number. So, $n$ must be even. From  Corollary ~\ref{cor_ineq} and Corollary \ref{imp_cor}, we obtain,
\begin{equation*}
    \Phi_n(a,b)=
    \Phi_{q-1}(a,b) \geqslant
    \frac{\Phi_{\frac{q-1}{2}}(a^2,b^2)}{\Phi_2(a,b)}\geqslant
    \frac{(a^2-b^2)^{\phi(\frac{q-1}{2})}}{a+b}. 
\end{equation*}
We can further refine the inequality as follows:
\begin{equation*}
q =
    \Phi_n(a,b) \geqslant
    (a+b)^{\phi(\frac{q-1}{2}) - 1} \geqslant
    3^{\phi(\frac{q-1}{2}) - 1}.
\end{equation*}
We will show that, this inequality is satisfied only when $q\leqslant7$. From Lemma ~\ref{sqrtineq} we have $\phi(n) \geqslant \dfrac{1}{2} \sqrt{n}$. If we put this into the previous inequality, we get
\begin{equation*}
q \geqslant
 3^{\phi(\frac{q-1}{2}) - 1}   \geqslant
 3^{\frac{\sqrt{q-1} }{2} - 1}.
\end{equation*}
This is only possible when $q \leqslant 113$. By putting this back into the inequality we obtain 
\begin{equation*}
 3^5 >  q \geq    3^{\phi(\frac{q-1}{2})-1}. 
\end{equation*}
This holds only when $\phi(\frac{q-1}{2})\leqslant5$, which is only possible if $q-1$ has no prime divisors greater than $5$. By manually checking all the remaining possibilities of $q$, we can see that 
$$q \geq 3^{\phi(\frac{q-1}{2})-1}$$
is satisfied only when $q\leq 13$.
If we look up all the cases, we get $n=2,4,6,12,$ with only $n=12$ being new.
\end{enumerate}

Now we will determine all the triples $(a,b,n)$ such that there is no large Zsigmondy prime for $(a,b,n)$.
We will use the Lemma \ref{cyclotomic_cases} and Lemma \ref{large_zsigmondy_existence}  to analyze the cases. 
\begin{enumerate}[(i)]
    \item If $n=2$, and there is no large Zsigmondy prime for $(a,b,n)$, then no prime greater than $3$ can divide $a^n-b^n$ furthermore $9 \nmid a^n-b^n$. 
    Then $a+b= 2^s 3^t$ for non-negative integers $s,t$ such that
    $t=0,1$. The case $t=0$ is also an exceptional case of Zsigmondy's theorem. 
    \item  If $n=4$, then $\Phi_4(a,b)=a^2+b^2 \leqslant 10$ must hold. Furthermore we have $a^2+b^2 \in \{5,10 \}$. We can easily check that only possible values for $(a,b)$ are $(2,1)$ and $(3,1)$. 
     \item  If $n=6$, then $\Phi_6(a,b)=a^2-ab+b^2 \leqslant 21$ must hold. Furthermore we have $\Phi_6(a,b) \in \{7, 21 \}$. From this we get $(3,1),(3,2),(5,4)$ as suitable values of $(a,b)$. Also we have one exceptional case of Zsigmondy's Theorem here when $(a,b)=(2,1)$. 
     
      \item  If $n=10$, then $\Phi_{10}(a,b)=a^4-a^3b+a^2b^2-ab^2+b^4 \leqslant 55$ must hold. Furthermore we have $\Phi_{10}(a,b) \in \{11, 55 \}$. We can easily check that only possible value for $(a,b)$ is $(2,1)$.
     
       \item  If $n=12$, then $\Phi_{12}(a,b)=a^4-a^2b^2+b^4 \leqslant 39$ must hold. Furthermore we have $\Phi_{12}(a,b) \in \{13, 39 \}$. We can easily check that only possible value for $(a,b)$ is $(2,1)$.
       
       \item  If $n=18$, then $\Phi_6(a,b)=a^6-a^3b^3+b^6 \leqslant 57$ must hold. Furthermore we have $\Phi_{18}(a,b) \in \{19, 57 \}$. We can easily check that only possible value for $(a,b)$ is $(2,1)$.
       
\end{enumerate}

\end{mainproof}


\bibliographystyle{amsplain}

\end{document}